\documentclass[10pt]{article}

\usepackage{amsthm,amsmath,amssymb}
\usepackage{bbm}
\newtheorem{lemma}{Lemma}
\newtheorem{theorem}{Theorem}

\newtheorem{proposition}{Proposition}

\def\ER{Erd\H{o}s-R\'enyi }
\def\PP{{\mathbb P}}
\def\EE{{\mathbb E}}

\def\R{{\mathbb{R}}}
\def\N{{\mathbb N}}

\def\Z{{\mathbb Z}}
\def\GG{{\mathbb G}}

\def\cM{{\mathcal M}}

\def\cH{{\mathcal H}}

\def\cP{{\mathcal P}}
\def\cW{{\mathcal W}}

\def\ind{{\mathbbm{1}}}

\def\f{\mathfrak f}
\def\m{\mathfrak m}

\title{The interpolation method for random graphs with prescribed degrees}
\author{Justin Salez}

\begin{document}
\maketitle

\begin{abstract}
We consider large random graphs with prescribed degrees, such as those generated by the configuration model. In the regime where the empirical degree distribution approaches a limit $\mu$ with finite mean, we establish the systematic convergence of a broad class of graph parameters that includes in particular the independence number, the maximum cut size and the log-partition function of the antiferromagnetic Ising and Potts models. The corresponding limits are shown to be Lipschitz and concave functions of $\mu$. Our work extends the applicability of the celebrated interpolation method, introduced in the context of spin glasses, and recently related to the fascinating problem of right-convergence of sparse graphs.\vspace{0.1cm} \\
\textbf{Keywords:} interpolation method; graph parameters; configuration model.\vspace{0.1cm}\\
\textbf{2010 MSC:} 60C05, 05C80, 82-08.
\end{abstract}

\section{Introduction}

\paragraph{Background.} A decade ago, Guerra and Toninelli \cite{GT02} introduced a powerful method to prove the existence of an infinite volume limit for the normalized log-partition function of the celebrated \emph{Sherrington-Kirkpatrick model}. The argument is based on an ingenious interpolation scheme which allows a system of size $n$ to be compared with two similar but independent systems of sizes $n_1$ and $n_2$ respectively, where $n_1+n_2=n$. The quantity of interest turns out to be sub-additive with respect to $n$, hence convergent once divided by $n$. This technique was then transferred from fully-connected models (complete graph) to their \emph{diluted} counterparts (sparse random graphs), where each particle only interacts with a finite, random number of neighbours. See in particular \cite{FL03} for the \ER case, and \cite{FLT03} for arbitrary degree distributions. 

In a recent breakthrough  \cite{interp}, the applicability of the interpolation method was extended to a variety of important models including, among others, the \emph{Ising model}, the \emph{Potts model} and the \emph{hard-core model}. As a special case, the long-conjectured convergence of the independence ratio of sparse \ER and $d-$regular random graphs was confirmed (see also the recent preprint \cite{sly}, where the limit is explicitly determined when $d$ is  large enough). The sub-additivity inherent to all those models was subsequently shown to follow from a single convexity property \cite{Gamarnik}, thereby shedding new light on the fascinating question of  \emph{right-convergence} (i.e. generic convergence of log-partition functions) of sparse random graphs \cite{BCKL,sparse,LD}. 

The aim of the present paper is to extend to random graphs with an arbitrary degree sequence the results obtained in \cite{interp,Gamarnik} for the \ER and $d-$regular random graphs. This substantial generalization allows us to investigate the general properties of the infinite volume limits when regarded as functions of the asymptotic degree distribution. Our interpolation scheme is rather simple, and the class of graph parameters to which our result applies is not \emph{a priori} restricted to log-partition functions or their zero-temperature limits. 

\paragraph{Graph parameters.} All graphs considered here are finite and undirected, with loops and multiple edges allowed. 
A \emph{graph parameter} is a real-valued function $\f$ defined on graphs, that is invariant under isomorphism. We shall restrict our attention to graph parameters satisfying the following properties. 
\begin{itemize}
\item \emph{Additive}: if $G$ is the vertex-disjoint union of $G_1$ and $G_2$, then 
$$\f(G) =  \f(G_1)+\f(G_2).$$
\item \emph{Lipschitz}: there is $\kappa<\infty$ so that for any graph $G=(V,E)$ and $(i,j)\in V^2$, 
\begin{eqnarray*}
\label{h2}
\left|\Delta^G_{ij}\right|\leq\kappa,& \textrm{ where } &\Delta^G_{ij}=\f(G + ij)-\f(G).
\end{eqnarray*} 
($G + ij$ is the graph obtained by adding a new edge between $i$ and $j$).
\item \emph{Concave}: $\{\Delta^G_{ij}\}_{(i,j)\in V\times V}$  is \emph{conditionally negative semidefinite} (\textsc{cnd}), i.e.  
\begin{eqnarray*}
\label{cnd}
\sum_{i\in V}x_i=0 & \Longrightarrow & \sum_{(i,j)\in V\times V}\Delta^G_{ij}x_ix_j\leq 0.
\end{eqnarray*}
\end{itemize}

This relaxed form of negative semi-definiteness is slightly less restrictive than the one imposed in \cite{Gamarnik}. \textsc{cnd} matrices are well-studied due to their intimate connection with infinite divisibile matrices. We gather here some of their most useful properties, and refer the reader to \cite[Chapter 4]{bapat} for more details. 
\begin{enumerate}
\item $\{\Delta_{ij}\}$ is \textsc{cnd} if and only if $\{\alpha_i+\alpha_j-\Delta_{ij}\}$ is positive semidefinite for some $\{\alpha_i\}\in\R^V$. Another equivalent condition is the infinite divisibility of $\{e^{-\Delta_{ij}}\}$, i.e.  $\{e^{-\lambda\Delta_{ij}}\}$ is positive definite for all $\lambda>0$.
\item 
\label{convex}
The \textsc{cnd} matrices $\{\Delta_{ij}\}_{(i,j)\in V\times V}$ form a convex cone.
\item 
\label{project}
If $\{\Delta_{ij}\}_{(i,j)\in S\times S}$ is \textsc{cnd} then so is $\{\Delta_{\sigma(i),\sigma(j)}\}_{(i,j)\in V\times V}$ for any $\sigma\colon V\to S$.
\item 
\label{log}
A sufficient condition for $\{\Delta_{ij}\}$ to be \textsc{cnd} is that $\{e^{\Delta_{ij}}\}$ is \textsc{cnd} (combine \cite[Theorem 4.4.4]{bapat} with \cite[Corollary 4.1.5]{bapat}).
\end{enumerate}

\paragraph{Examples.}
Many important graph parameters (or their negative) belong to the above class. Here are a few examples.
\begin{itemize}
\item \emph{Number of connected components}: the increment matrix is simply   $\Delta^G=-\sum\ind^{\phantom{\top}}_S\ind^\top_S$, where the sum runs over the connected components $S$ of $G$. 
\item \emph{Independence number}: here $\Delta^G_{ij}=1-\left(\ind^{\phantom{\top}}_S\ind^\top_S\right)_{ij}$, where the set $S\subseteq V$ is the intersection of all maximum independent sets on $G$. 
\item \emph{Maximum cut size}: $\Delta^G_{ij}=1-\prod(\ind^{\phantom{\top}}_S\ind^\top_S
+\ind^{\phantom{\top}}_{\overline{S}}\ind^\top_{\overline{S}})_{ij}$, where $\Pi$ runs over all maximum cuts $(S,\overline S)$ (entry-wise product preserves positive definiteness). 
\item \emph{Log-partition functions}: fix a finite set $S$, a map $h\colon S\to (0,+\infty)$  and a symmetric map $J\colon S\times S \to (0,+\infty)$, and consider the graph parameter
\begin{eqnarray*}
\f(G)& := & \log\left(\sum_{\sigma\in S^V}w(\sigma)\right)\textrm{ where }w(\sigma)=\prod_{i\in V}h(\sigma_i)\prod_{ij\in E}J(\sigma_i,\sigma_j).
\end{eqnarray*}
Then $\f$ is easily seen to be additive and Lipschitz, and a sufficient condition for it to be concave is that the matrix $J$ is  \textsc{cnd}. Indeed, (iii) ensures that  $J^\sigma:=\{J(\sigma_i,\sigma_j)\}_{(i,j)\in V\times V}$ is \textsc{cnd} for all  $\sigma\in S^V$ and (ii) then implies that $$\frac{\sum_{\sigma\in S^V}w(\sigma)J^\sigma}{\sum_{\sigma\in S^V}w(\sigma)}$$ is \textsc{cnd}. But this is  exactly $\{e^{\f(G+ij)-\f(G)}\}$, and (iv) allows to conclude.

In particular, the log-partition functions of the \emph{Ising model} ($S=\{-1,+1\}$, $J(s,t)=e^{-\beta st}$) and  \emph{Potts model} ($S=\{1,\ldots,q\}$, $J(s,t)=\ind_{s\neq t}+e^{-\beta}\ind_{s=t}$) are additive, Lipschitz and concave graph parameters for all $\beta\geq 0$.
\end{itemize}

\paragraph{Result.} The present paper is concerned with the asymptotic behaviour of such graph parameters when evaluated on large random graphs with prescribed degrees. 
For each $n\geq 1$, we let $\GG_{d_n}$ denote a random graph on $V=\{1,\ldots,n\}$ generated by the \emph{configuration model} \cite{Bol80,Hof13} with degrees $d_n=\{d_n(i)\}_{1\leq i\leq n}$  (see section \ref{sec:main} for the precise definition). 
We assume that the sequence $\{d_n\}_{n\geq 1}$ approaches a probability measure $\mu$ on $\N$ with mean $\overline\mu<+\infty$, in the sense that 
\begin{eqnarray}
\label{h4}
\forall k\in\N,\qquad \frac{1}{n}\sum_{i=1}^n\ind_{\{d_n(i)=k\}} & \xrightarrow[n\to\infty]{}& \mu(k)\\
\label{h5}
\frac{1}{n}\sum_{i=1}^n d_n(i) & \xrightarrow[n\to\infty]{} &\overline{\mu}.
\end{eqnarray}
In other words, the empirical measure $\frac{1}{n}\sum_{i=1}^n\delta_{d_n(i)}$ converges to $\mu$ in the \emph{Wasserstein space} $\cP_1(\N)$. This is the space of probability measures on $\N$ with finite mean, equipped with the Wasserstein distance 
\begin{eqnarray*}
{\cW}\left(\mu,\mu'\right) 
& = & \sum_{i=1}^\infty\left|\sum_{k=i}^\infty\left(\mu(k)-\mu'(k)\right)\right|.
\end{eqnarray*}
We refer the reader to the books  \cite[Chapter 2]{W1} or \cite[Chapter 6]{villani} for more details on Wasserstein spaces $\cP_p({\cal X})$ and  many alternative expressions for $\cW$. 

\begin{theorem}\label{th:main}Every additive, Lipschitz, concave graph parameter admits an ``infinite volume limit" $\Psi\colon\cP_1(\N)\to\R$ in the following sense: for any $\mu\in\cP_1(\N)$ and any $\{d_n\}_{n\geq 1}$ satisfying (\ref{h4})-(\ref{h5}), we have the almost-sure convergence
\begin{eqnarray}
\label{cvmain}
\frac{\f(\GG_{d_n})}{n} & \xrightarrow[n\to\infty]{} & \Psi(\mu).
\end{eqnarray}
Moreover,  $\Psi$ is Lipschitz and concave: for any $\mu,\mu'\in\cP_1(\N)$,
\begin{eqnarray}
\label{lipschitz}
\left|\Psi(\mu)-\Psi(\mu')\right| & \leq & 2\kappa\cW(\mu,\mu')\\
\label{concave}
\theta\Psi(\mu)+(1-\theta)\Psi(\mu')  & \leq  & \Psi(\theta\mu+(1-\theta)\mu') \qquad (0\leq\theta\leq 1).
\end{eqnarray}
\end{theorem}

\paragraph{The \textsc{IID} case.} A common setting consists in taking $d_n=(\Delta_1,\ldots,\Delta_n)$, where $\{\Delta_i\}_{i\geq 1}$ are \textsc{iid} samples from a target degree distribution $\mu\in\cP_1(\N)$. We denote  by $\GG^\textsc{iid}_{\mu,n}$ the resulting doubly random graph. The result (\ref{cvmain}) applies, since  (\ref{h4})-(\ref{h5}) hold almost-surely by the strong law of large numbers. In fact, the convergence
\begin{eqnarray}
\label{sllnw}
\cW\left(\frac{1}{n}\sum_{i=1}^n\delta_{\Delta_i},\mu\right)& \xrightarrow[n\to\infty]{} & 0,
\end{eqnarray}
holds almost-surely and in $L^1$, see \cite[Theorem 2.14]{W1} and  \cite[Theorem 3.5]{W1}.

\paragraph{Simple graphs.} Under assumption (\ref{h4}), a sufficient condition for (\ref{h5}) is
\begin{equation}
\label{h6}
\sup_{n\geq 1}\frac{1}{n}\sum_{i=1}^n d^2_n(i)<\infty.
\end{equation}
Under this condition and if $\sum_{i=1}^n d_n(i)$ is even, the probability that $\GG_{d_n}$ is simple remains bounded away from $0$ as $n\to\infty$, see \cite{Jan09,janson2}. Moreover, conditionally on being simple, $\GG_{d_n}$ is uniformly distributed on the set of all simple graphs with degrees $\{d_n(i)\}_{1\leq i\leq n}$. Consequently, the convergence (\ref{cvmain}) also applies to uniform simple graphs with degrees $\{d_n(i)\}_{1\leq i\leq n}$. In the sparse regime, the \ER random graph and the more general \emph{rank-one inhomogeneous random graph} \cite{IRG, Hof13} have degree sequences which satisfy almost-surely assumptions (\ref{h4}) and (\ref{h6}). Moreover, conditionally on the degree sequence, their distribution is uniform. Thus, the conclusion of Theorem \ref{th:main} applies to those models as well.

\paragraph{Extensions.} By linearity, the convergence (\ref{cvmain}) extends to any linear combination of additive, Lipschitz, concave graph parameters. Such parameters remain  additive and Lipschitz, and it is perhaps natural to ask the following:
\begin{center}
\emph{
Does the convergence (\ref{cvmain}) hold for any additive Lipschitz graph parameter?
}
\end{center}
Note that a positive answer would in particular imply  \cite[Conjecture 1]{Gamarnik}.

\section{Proof outline}

Our main ingredient is the following inequality, with $\varphi(x)=7\kappa\,\sqrt{x\ln(1+x)}$. 
\begin{proposition}
\label{pr:main}
Let $A,B$ be finite disjoint sets and let $d\colon A\cup B\to\N$. Write $d\restriction A$, $d\restriction B$  for the restrictions of $d$ to $A$, $B$. Then,
\begin{eqnarray*}
\EE\left[\f\left(\GG_{d\restriction A}\right)\right]+\EE\left[\f\left(\GG_{d\restriction B}\right)\right] & \leq & \EE\left[\f\left(\GG_{d}\right)\right]+\varphi\left(\frac{1}{2}\sum_{i\in A\cup B}d(i)\right).
\end{eqnarray*}
\end{proposition}
As this holds for any degrees $\{d(i)\}_{i\in A\cup B}$, we may fix $\mu\in\cP_1(\N)$ and average it against $\mu^{\otimes A\cup B}$. Since $\varphi$ is concave, Jensen's inequality yields
\begin{eqnarray*}
\EE\left[\f(\GG^\textsc{iid}_{\mu,|A|})\right]+\EE\left[\f(\GG^\textsc{iid}_{\mu,|B|})\right]& \leq & \EE\left[\f(\GG^\textsc{iid}_{\mu,|A|+|B|})\right]+\varphi\left(\frac{\overline{\mu}}{2}(|A|+|B|)\right).
\end{eqnarray*}
By a classical result of De Bruijn and Erd\H{o}s \cite[Theorem 23]{Bruijn}, this near super-additivity suffices to guarantee the existence of the limit
\begin{eqnarray}
\label{eq:cviid}
\Psi(\mu)& := & \lim_{n\to\infty}\frac{\EE\left[\f\left(\GG^\textsc{iid}_{\mu,n}\right)\right]}{n}\in\R\cup\{+\infty\}.
\end{eqnarray} 
In our case we must have $\Psi(\mu)<\infty$, since the additive and Lipschitz properties of $\f$ easily imply that $\f(G)=O(|V|+|E|)$ uniformly over all graphs $G=(V,E)$. 

Our second ingredient is the following simple result, which quantifies the intuition that $\EE\left[\f(\GG_{d'})\right]$ should be close to $\EE\left[\f(\GG_{d})\right]$ whenever $d'$ is close to $d$.
\begin{proposition}
\label{pr:compare}
For any $d,d'\colon \{1,\ldots,n\}\to\N$, 
\begin{eqnarray*}
\left|\frac{\EE\left[\f(\GG_{d})\right]}n-\frac{\EE\left[\f(\GG_{d'})\right]}n\right| & \leq & 2\kappa \cW\left(\frac{1}{n}\sum_{i=1}^n\delta_{d(i)},\frac{1}{n}\sum_{i=1}^n\delta_{d'(i)}\right).
\end{eqnarray*}
\end{proposition}
Let us apply this when $\{d'(i)\}_{1\leq i\leq n}$ are \textsc{iid} samples from $\mu$. Recalling the $L^1$ convergence (\ref{sllnw}) and the assumption that $\{d_n\}_{n\geq 1}$ approaches $\mu$, we obtain
\begin{eqnarray*}
\left|\frac{\EE\left[\f\left(\GG_{d_n}\right)\right]}{n}-\frac{\EE\left[\f\left(\GG^\textsc{iid}_{\mu,n}\right)\right]}{n} \right|
& \xrightarrow[n\to\infty]{} & 0,
\end{eqnarray*}
by the triangle inequality.  
In view of (\ref{eq:cviid}), we may now conclude that
\begin{eqnarray}
\label{cv:expect}
\frac{\EE[\f(\GG_{d_n})]}{n}& \xrightarrow[n\to\infty]{}& \Psi(\mu).
\end{eqnarray}

Finally, since  $\f$ is Lipschitz, a now-standard application of Azuma-Hoeffding's inequality ensures that $\f(\GG_d)$ is exponentially concentrated : for any $\varepsilon>0$,
\begin{eqnarray}
\label{concentration}
\PP\left(\left|{\f(\GG_{d})}-{\EE[\f(\GG_{d})]}\right|\geq \varepsilon\right)& \leq & \exp\left(-\frac{\varepsilon^2}{4\kappa^2\sum_id(i)}\right).
\end{eqnarray}
See \cite[Theorem 2.19]{Wormald} for a proof when $d$ is constant and \cite[Corollary 3.27]{Bordenave} for the general case. 
In view of (\ref{h5}), Borel-Cantelli's Lemma ensures that 
\begin{eqnarray*}
\left|\frac{\f(\GG_{d_n})}{n}-\frac{\EE[\f(\GG_{d_n})]}{n}\right| & \xrightarrow[n\to\infty]{}& 0, 
\end{eqnarray*}
almost-surely under any coupling of the random graphs $\left\{\GG_{d_n}\right\}_{n\geq 1}$. Combining this with (\ref{cv:expect}) concludes the proof of (\ref{cvmain}). The Lipschitz continuity (\ref{lipschitz}) follows by passing to the limit in Proposition \ref{pr:compare} along sequences $\{d_n\}_{n\geq 1},\{d_n'\}_{n\geq 1}$ that approach $\mu,\mu'$ in the sense of $(\ref{h4})-(\ref{h5})$. Since the concatenation of $d_n$ and $d'_n$ approaches $\frac{\mu+\mu'}{2}$, we may also pass to the limit in Proposition \ref{pr:main} to obtain (\ref{concave}) when $\theta=\frac{1}{2}$. This mid-point concavity implies concavity, since $\Psi$ is continuous.  The remainder of the paper is devoted to the proofs of Proposition \ref{pr:main}  and \ref{pr:compare}.

\section{Proof of Proposition \ref{pr:main}}
\label{sec:main}

Throughout this section, we fix a finite set $V$ and a function $d\colon V\to\N$. Form a set $\cH$ of half-edges by ``attaching" $d(i)$ half-edges with each end-point $i\in V$:
$$\cH:=\bigcup_{i\in V}\left\{(i,1),\ldots,(i,d(i))\right\}$$
 A (partial) \emph{matching} $\m$ of $\cH$ is a collection of pairwise disjoint $2-$element subsets of $\cH$. Such a matching naturally induces a graph $G[\m]$ on $V$ by interpreting a pair of matched half-edges as an edge between the corresponding end-points.  By definition,  $\GG_d$ is the random graph induced by a uniformly chosen random maximal matching on $\cH$. Now fix a bipartition $V=A\cup B$, and define
\begin{eqnarray}
\label{sumnotation}
d(A)=\sum_{i\in A}d(i) & \textrm{ and } & d(B)=\sum_{i\in B}d(i).
\end{eqnarray}
 For $\left(\alpha,\beta,\gamma\right)\in\N^3$, let $\cM({\alpha,\beta,\gamma})$ denote the set of all matchings of $\cH$ containing  
\begin{itemize}
\item $\alpha$ edges with both end-points in $A$ 
\item $\beta$ edges with both end-points in $B$
\item $\gamma$ edges with one end-point in $A$ and the other in $B$ (called cross-edges).
\end{itemize}
Note that $\cM({\alpha,\beta,\gamma})\neq\emptyset$ only if
$
2\alpha+\gamma\leq d(A)$ and $2\beta+\gamma\leq d(B)$. When this condition holds, we call the triple $(\alpha,\beta,\gamma)$ feasible, and we define
\begin{eqnarray*}
F(\alpha,\beta,\gamma) & \textrm{ = } & \frac{1}{|\cM({\alpha,\beta,\gamma})|}\sum_{\m\in\cM({\alpha,\beta,\gamma})}\f(G[\m]).
\end{eqnarray*}
In other words, $F(\alpha,\beta,\gamma)$ is the expectation of $\f(G[\m])$ when $\m$ is uniform on $\cM({\alpha,\beta,\gamma})$. To connect this with the configuration model, observe that conditionally on its number $\gamma$ of cross-edges, a uniformly chosen maximal matching $\m$ on $\cH$ is uniformly distributed in $\cM\left({\left\lfloor\frac{d(A)-\gamma}{2}\right\rfloor,\left\lfloor\frac{d(B)-\gamma}{2}\right\rfloor,\gamma}\right)$. Thus, $F\left(\left\lfloor\frac{d(A)-\gamma}{2}\right\rfloor,\left\lfloor\frac{d(B)-\gamma}{2}\right\rfloor,\gamma\right)$ is the conditional expectation of $\f(\GG_d)$ given the number  $\gamma$ of cross-edges. 
On the other-hand, since $\f$ is additive,
$$F\left(\left\lfloor\frac{d(A)}{2}\right\rfloor,\left\lfloor\frac{d(B)}{2}\right\rfloor,0\right)=\EE\left[\f(\GG_{d\restriction A})\right]+\EE\left[\f(\GG_{d\restriction B})\right].$$
Therefore, Proposition \ref{pr:main} is a consequence of the following stronger result, to the proof of which this whole section is devoted. 
\begin{proposition}
\label{pr:global}
For any non-negative integer $\gamma\leq d(A)\wedge d(B)$,
\begin{eqnarray*}
F\left(\left\lfloor\frac{d(A)}{2}\right\rfloor,\left\lfloor\frac{d(B)}{2}\right\rfloor,0\right)  \leq  F\left(\left\lfloor\frac{d(A)-\gamma}{2}\right\rfloor,\left\lfloor\frac{d(B)-\gamma}{2}\right\rfloor,\gamma\right)
+ \varphi(\gamma).
\end{eqnarray*}
\end{proposition}

Given a matching $\m$ on $\cH$, one can create a larger matching $\m'\supset \m$ by adding to $\m$ a uniformly chosen pair of distinct unmatched half-edges (provided they exist). Restricting the choice to half-edges whose end-point is in $A$, or in $B$, or to pairs in which one end-point is in $A$ and the other in $B$ defines what we call a random $A-$pairing, $B-$pairing or cross-pairing. These can be performed sequentially to sample $\cM(\alpha,\beta,\gamma)$ uniformly, as shown by the following Lemma.
\begin{lemma}
\label{lm:pairing}
Let $\m$ be uniformly distributed on $\cM(\alpha,\beta,\gamma)$. Conditionally on $\m$, make a random $A$ (resp. $B$, resp.  cross) pairing. Then the result $\m'$ is uniformly distributed on $\cM(\alpha+1,\beta,\gamma)$  (resp. $\cM(\alpha,\beta+1,\gamma)$, resp. $\cM(\alpha,\beta,\gamma+1)$). 
\end{lemma}
\begin{proof}
Every $m\in \cM(\alpha,\beta,\gamma)$ admits ${d(A)-2\alpha-\gamma\choose 2}$  allowed $A-$pairings, each producing a distinct $m'\supset m$ in $\cM(\alpha+1,\beta,\gamma)$. By uniformity, it follows that $\PP(\m'=m')$ is proportional to the number of matchings $m\in\cM(\alpha,\beta,\gamma)$ such that $m\subset m'$. But this is exactly $\alpha+1$, independently of $m'\in\cM(\alpha+1,\beta,\gamma)$. The argument for $B-$pairings and cross-pairings is similar. 
\end{proof}
We now exploit this useful observation to establish two key properties of $F$.
\begin{lemma}[Lipschitz continuity]
\label{lm:lipschitz}
For any feasible $(\alpha,\beta,\gamma)$ and $(\alpha',\beta',\gamma')$,
\begin{eqnarray*}
\left|F(\alpha,\beta,\gamma)-F(\alpha',\beta',\gamma')\right| & \leq & \kappa\left(|\alpha-\alpha'|+|\beta-\beta'|+|\gamma-\gamma'|\right)
\end{eqnarray*}
\end{lemma}
\begin{proof} It is sufficient to prove this when the triples differ by $1$ at a single coordinate. Let us treat only the case $(\alpha',\beta',\gamma')=(\alpha+1,\beta,\gamma)$, the proof for the other cases being similar. Let $\m$ be uniform in $\cM(\alpha,\beta,\gamma)$, and let $\m'$ be obtained from $\m$ by a random $A-$pairing. Then $G[\m']$ differs from $G[\m]$ by exactly one edge, so the Lipschitz assumption guarantees that a-s, $$\f(G[\m])-\kappa\leq \f(G[\m'])\leq\f(G[\m])+ \kappa.$$
But $\m'$ is uniformly distributed on $\cM(\alpha+1,\beta,\gamma)$ by Lemma \ref{lm:pairing}, so taking expectations above yields precisely the desired result. 
\end{proof}

\begin{lemma}
\label{superadd}
Local super-additivity: for $\delta\geq 2$, if $(\alpha,\beta,\gamma+\delta)$ is feasible then 
\begin{eqnarray*}
\frac{F(\alpha+1,\beta,\gamma)+F(\alpha,\beta+1,\gamma)}{2} & \leq  & F(\alpha,\beta,\gamma+1)+ \frac{2\kappa}{\delta}
\end{eqnarray*}
\end{lemma}

\begin{proof}Fix  $\m\in\cM(\alpha,\beta,\gamma)$. Let $\m'$ be obtained from $\m$ by  a random cross-pairing, and let $\m''$ be obtained from $\m$ by flipping a fair coin and making a random $A-$pairing or $B-$pairing accordingly. We will prove that 
\begin{eqnarray*}
\EE\left[\f(G[\m''])\right] -\EE\left[\f(G[\m'])\right] & \leq & \frac{\kappa}{d(A)-2\alpha-\gamma}+\frac{\kappa}{d(B)-2\beta-\gamma}.
\end{eqnarray*}
The assumption ensures that the right-hand side is at most $2\kappa/\delta$, and averaging over all $\m\in\cM(\alpha,\beta,\gamma)$ implies the result, by Lemma \ref{lm:pairing}. 
  Write $c(i)$ for the number of unpaired half-edges attached to $i\in V$ in $\m$, and define $c(A),c(B)$ as in  (\ref{sumnotation}). Set also $\Delta_{ij}=\f(G[\m+ij])-\f(G[\m])$. With this notation, we have
\begin{eqnarray*}
\EE\left[\f(G[\m''])-\f(G[\m])\right] & = & \frac{1}{2}\sum_{(i,j)\in A\times A}\Delta_{ij}\frac{c(i)\left(c(j)-\ind_{i=j}\right)}{c(A)\left(c(A)-1\right)}\\ & + & \frac{1}{2}\sum_{(i,j)\in B\times B}\Delta_{ij}\frac{c(i)\left(c(j)-\ind_{i=j}\right)}{c(B)\left(c(B)-1\right)}\\
\EE\left[\f(G[\m'])-\f(G[\m])\right]  & = & \frac{1}{2}\sum_{(i,j)\in A\times B}\Delta_{ij}\frac{c(i)c(j)}{c(A)c(B)} + \frac{1}{2}\sum_{(i,j)\in B\times A}\Delta_{ij}\frac{c(i)c(j)}{c(B)c(A)}.
\end{eqnarray*}
We may thus decompose the difference $\EE[\f(G[\m''])] - 
\EE[\f(G[\m'])]$ as  
\begin{eqnarray}
\label{three}
\frac 12\sum_{(i,j)\in V\times V}\Delta_{ij}x_ix_j + \frac 12\sum_{(i,j)\in A\times A}\Delta_{ij}z_{ij}+\frac 12\sum_{(i,j)\in B\times B}\Delta_{ij}z_{ij}
\end{eqnarray}
where we have set 
\begin{eqnarray*}
x_i \ = \
\begin{cases}
\frac{c(i)}{c(A)}& \textrm{ if }i\in A\\
\smallskip\\
-\frac{c(i)}{c(B)} & \textrm{ if }i\in B
\end{cases}
& \textrm{  }& 
z_{ij} \ = \ \begin{cases}
\frac{c(i)\left(c(j)-c(A)\ind_{i=j}\right)}{c(A)c(A)\left(c(A)-1\right)}& \textrm{ if }(i,j)\in A\times A\\ \smallskip \\
\frac{c(i)\left(c(j)-c(B)\ind_{i=j}\right)}{c(B)c(B)\left(c(B)-1\right)}& \textrm{ if }(i,j)\in B\times B
\end{cases}
\end{eqnarray*}
Now, the first term in (\ref{three}) is non-positive since $\sum_{i}x_i=0$ and $\{\Delta_{ij}\}$ is \textsc{cnd} ($\f$ is concave). For the second term, note that $|\Delta_{ij}|\leq\kappa$ ($\f$ is Lipschitz), so that
\begin{eqnarray*}
\frac 12\sum_{(i,j)\in A\times A}\Delta_{ij}z_{ij} 
& \leq & \frac {\kappa}2\sum_{(i,j)\in A\times A}|z_{ij}|\\
& \leq & \kappa\sum_{i\in A}\frac{c(i)\left(c(A)-c(i)\right)}{c(A)c(A)\left(c(A)-1\right)}\\
& \leq & \frac{\kappa}{c(A)},
\end{eqnarray*}
where we have used the inequality $k(n-k)\leq k(n-1)$ valid for any integers $0\leq k\leq n$. Replacing $A$ with $B$ yields the bound $\frac{\kappa}{c(B)}$ for the third term.
\end{proof}

We may now deduce Proposition \ref{pr:global} from the above two properties of $F$. 

\begin{proof}[Proof of Proposition \ref{pr:global}]
The claim is trivial when $\gamma$ is small. Indeed, 
$$\left|\left\lfloor\frac{d(A)}{2}\right\rfloor-\left\lfloor\frac{d(A)-\gamma}{2}\right\rfloor\right|
+
\left|\left\lfloor\frac{d(B)}{2}\right\rfloor-\left\lfloor\frac{d(B)-\gamma}{2}\right\rfloor\right|+|0-\gamma|\leq 2\gamma+1
,$$
so Lemma \ref{lm:lipschitz} guarantees that
$$F\left(\left\lfloor\frac{d(A)}{2}\right\rfloor,\left\lfloor\frac{d(B)}{2}\right\rfloor,0\right)  \leq  F\left(\left\lfloor\frac{d(A)-\gamma}{2}\right\rfloor,\left\lfloor\frac{d(B)-\gamma}{2}\right\rfloor,\gamma\right)+\kappa(2\gamma+1).$$
This implies the claim as long as $2\gamma+1\leq 7\sqrt{\gamma\ln(1+\gamma)}$, i.e. $\gamma\in\{1,\ldots,46\}$. We now assume that $\gamma\geq 47$. Let $\delta\in\N$ to be chosen later, such that $2\leq \delta\leq {\gamma}/{2}$. 
Let $\{S_t\}_{t\in\N}$ be a simple random walk on $\Z$ started at $0$. Set $\tau=\gamma-2\delta$ and for every $t\in\{0,\ldots,\tau\}$, consider the random triple $\left({\boldsymbol\alpha}_{t},{\boldsymbol\beta}_{t},{\boldsymbol\gamma}_{t}\right)$ defined by 
\begin{eqnarray*}
{\boldsymbol\alpha}_t  = \left\lfloor\frac{d(A)-\gamma}{2}\right\rfloor+\frac{t+S_{t}}{2},\ \
{\boldsymbol\beta}_t  =  \left\lfloor\frac{d(B)-\gamma}{2}\right\rfloor+\frac{t-S_{t}}{2},\ \
{\boldsymbol\gamma}_t  =  \tau-t.
\end{eqnarray*}
Thus, conditionally on ${\cal F}_t=\sigma\left(S_0,\ldots,S_t\right)$, the triple $\left(\alpha_{t+1},\beta_{t+1},\gamma_{t+1}\right)$ is obtained from $\left(\alpha_{t},\beta_{t},\gamma_{t}\right)$ by decrementing the last coordinate and incrementing either the first or the second coordinate, with probability half each. Moreover, it is immediate to check that $\left({\boldsymbol\alpha}_{t},{\boldsymbol\beta}_{t},{\boldsymbol\gamma}_{t}+\delta\right)$ is feasible for all $t< \tau \wedge T$, where
$$T:=\inf\left\{t\in\N\colon |S_t|> \delta\right\}.$$ 
Therefore, Lemma \ref{superadd} guarantees that the stochastic process $\{Z_t\}_{0\leq t \leq \tau}$ defined by
$Z_t:=F\left({\boldsymbol\alpha}_{t\wedge T},{\boldsymbol\beta}_{t\wedge T},{\boldsymbol\gamma}_{t\wedge T}\right)$ satisfies 
\begin{eqnarray*}
\EE\left[Z_{t+1}|{\cal F}_t\right] & \leq & Z_t+ \frac{2\kappa}{\delta}\ind_{\{t<T\}}.
\end{eqnarray*}
Taking expectations and summing over all $0\leq t <\tau$, we deduce that
\begin{eqnarray}
\label{ineq0}
\EE\left[Z_{\tau}\right]-Z_0 & \leq &  \frac{2\kappa\tau}{\delta}.
\end{eqnarray}
Now, since $Z_0=F\left(\left\lfloor\frac{d(A)-\gamma}{2}\right\rfloor,\left\lfloor\frac{d(B)-\gamma}{2}\right\rfloor,\gamma-2\delta\right)$, Lemma (\ref{lm:lipschitz}) yields
\begin{eqnarray}
\label{ineq1}
Z_0- F\left(\left\lfloor\frac{d(A)-\gamma}{2}\right\rfloor,\left\lfloor\frac{d(B)-\gamma}{2}\right\rfloor,\gamma\right) & \leq & 2\kappa\delta.
\end{eqnarray}
On the other hand, observing that almost-surely,
\begin{eqnarray*}
\left|\left\lfloor\frac{d(A)}{2}\right\rfloor-{\boldsymbol\alpha}_{\tau\wedge T}\right| 
+\left|\left\lfloor\frac{d(B)}{2}\right\rfloor-{\boldsymbol\beta}_{\tau\wedge T}\right|+\left|{\boldsymbol\gamma}_{\tau\wedge T}\right| 
& \leq & 2\tau\ind_{\tau\geq T}+2\delta+1,
\end{eqnarray*}
we may invoke Lemma \ref{lm:lipschitz} again to obtain
\begin{eqnarray*}
F\left(\left\lfloor\frac{d(A)}{2}\right\rfloor,\left\lfloor\frac{d(B)}{2}\right\rfloor,0\right)-Z_{\tau} & \leq &  \kappa\left(2\tau\ind_{\tau\geq T}+2\delta+1\right).
\end{eqnarray*}
Taking expectations yields
\begin{eqnarray}
\label{ineq2}
F\left(\left\lfloor\frac{d(A)}{2}\right\rfloor,\left\lfloor\frac{d(B)}{2}\right\rfloor,0\right)-\EE\left[Z_{\tau}\right] & \leq & \kappa\left(4\tau e^{-\frac{(\delta+1)^2}{2\tau}}+2\delta+1\right),
\end{eqnarray}
where we used the following classical consequence of Doob's maximal inequality:
\begin{eqnarray*}
\label{doob}
\PP\left(T\leq \tau\right)& = & \PP\left(\max_{0\leq t\leq \tau}|S_t|\geq \delta+1\right)
\ \leq \ 2 e^{-\frac{(\delta+1)^2}{2\tau}}.
\end{eqnarray*}
Adding up (\ref{ineq0})-(\ref{ineq1})-(\ref{ineq2}) and recalling that $\tau=\gamma-2\delta$, we finally arrive at
\begin{eqnarray*}
{F\left(\left\lfloor\frac{d(A)}{2}\right\rfloor,\left\lfloor\frac{d(B)}{2}\right\rfloor,0\right) - F\left(\left\lfloor\frac{d(A)-\gamma}{2}\right\rfloor,\left\lfloor\frac{d(B)-\gamma}{2}\right\rfloor,\gamma\right) } \\
 \leq 
 \kappa\left(\frac{2\gamma}{\delta}+4\delta+4\gamma e^{-\frac{(\delta+1)^2}{2\gamma}}\right).
\end{eqnarray*}
The choice  $\delta=\left\lfloor\sqrt{\gamma\ln(1+\gamma)}\right\rfloor$  yields the bound $c(\gamma)\kappa\sqrt{\gamma\ln(1+\gamma)}$, where
\begin{eqnarray*}
c(\gamma)=\frac{2}{\ln(1+\gamma)-\sqrt{\frac{\ln(1+\gamma)}{\gamma}}}+4+\frac{4}{\sqrt{\ln(1+\gamma)}}.
\end{eqnarray*}
This quantity decreases with $\gamma$, so $c(\gamma)\leq c(47)\approx 6.59<7$.
\end{proof}

\section{Proof of Proposition \ref{pr:compare}}
\label{sec:compare}

Let us first establish that 
\begin{eqnarray}
\label{eq:l1}
\left|\EE\left[\f(\GG_d)\right]-\EE\left[\f(\GG_{d'})\right]\right| & \leq & 2\kappa \sum_{i=1}^n|d(i)-d'(i)|.
\end{eqnarray}
By an immediate induction, we may restrict our attention to the case where $d(i)=d'(i){\ind}_{i=i_0}$
for some $i_0\in \{1,\ldots,n\}$. Recall that $\GG_d,\GG_{d'}$ can be realized as $G[\m],G[\m']$ where $\m,\m'$ are uniform maximal matchings on the corresponding sets of half-edges $\cH,\cH'$. But $\cH=\cH'\cup\{h\}$, where $h=(i_0,d(i_0))$ denotes the extra half-edge attached to $i_0$. We may thus couple $\m'$ to $\m$ as follows.
\begin{itemize}
\item If $|\cH|$ is even, then the matching $\m$ is perfect and we let $\m'$ denote the matching obtained by simply removing from $\m$ the pair containing $h$. 
\item If $|\cH|$ is odd, then there must be an unpaired half-edge in $\m$, and we let $\m'$ denote the matching obtained by exchanging it with $h$. 
\end{itemize}
In both cases, $\m'$ is uniformly distributed over the maximal matchings of $\cH'$. Moreover, $G[\m],G[\m']$ differ by at most two edges almost-surely, so that
$$\left|\f\left(G[\m]\right)-\f\left(G[\m']\right)\right|\leq 2\kappa.$$
Taking expectations yields (\ref{eq:l1}). Since $\f$ is invariant under graph isomorphism, the left-hand side of (\ref{eq:l1}) is invariant under reordering of $d,d'$. Consequently, we may choose a rearrangement that minimizes the right-hand side. It is classical that the choice $d(1)\leq\ldots\leq d(n)$ and $d'(1)\leq\ldots\leq d'(n)$ is optimal and satisfies 
\begin{eqnarray*}
\sum_{i=1}^n|d(i)-d'(i)| 
& = & n\cW\left(\frac{1}{n}\sum_{i=1}^n\delta_{d(i)},\frac{1}{n}\sum_{i=1}^n\delta_{d'(i)}\right).
\end{eqnarray*}
Re-injecting this into (\ref{eq:l1}) concludes the proof.

\bibliographystyle{plain}
\bibliography{draft}
\end{document}